\documentclass[a4paper,twoside,12pt]{article} 

\usepackage{amsmath,amssymb,amsthm,amsfonts}
\usepackage{microtype} 
\usepackage{cmbright}
\usepackage[T1]{fontenc}
\usepackage{lmodern}
\usepackage{cite} 
\usepackage{comment}

\usepackage[utf8]{inputenc} 
\usepackage[english]{babel}
\usepackage{enumerate}

\usepackage{graphicx} 
\usepackage{xcolor}
\usepackage{hyperref}
\definecolor{darkgreen}{rgb}{0,0.4,0}
\definecolor{BrickRed}{rgb}{0.65,0.08,0}
\hypersetup{colorlinks=true,linkcolor=blue,citecolor=darkgreen,filecolor=BrickRed,urlcolor=darkgreen}
\newcommand{\OEIS}[1]{\text{\href{https://oeis.org/#1}{{\small \tt #1}}}} 

\newtheorem{theorem}{Theorem}
\newtheorem{corollary}{Corollary}
\newtheorem{lemma}{Lemma}
 
\def\A{\alpha}
\def\B{\beta}

\def\ord{{\operatorname{ord}}}
\def\lcm{{\operatorname{lcm}}}

\usepackage{a4wide} 
\usepackage{fancyhdr}
\usepackage{comment}

\usepackage{MnSymbol} 

\begin{document}
\title{On the period mod $m$ of\\ polynomially-recursive sequences: a case study}

\date{March 5, 2019} 

\newcommand{\addorcid}[1]{\protect\includegraphics[height=3mm]{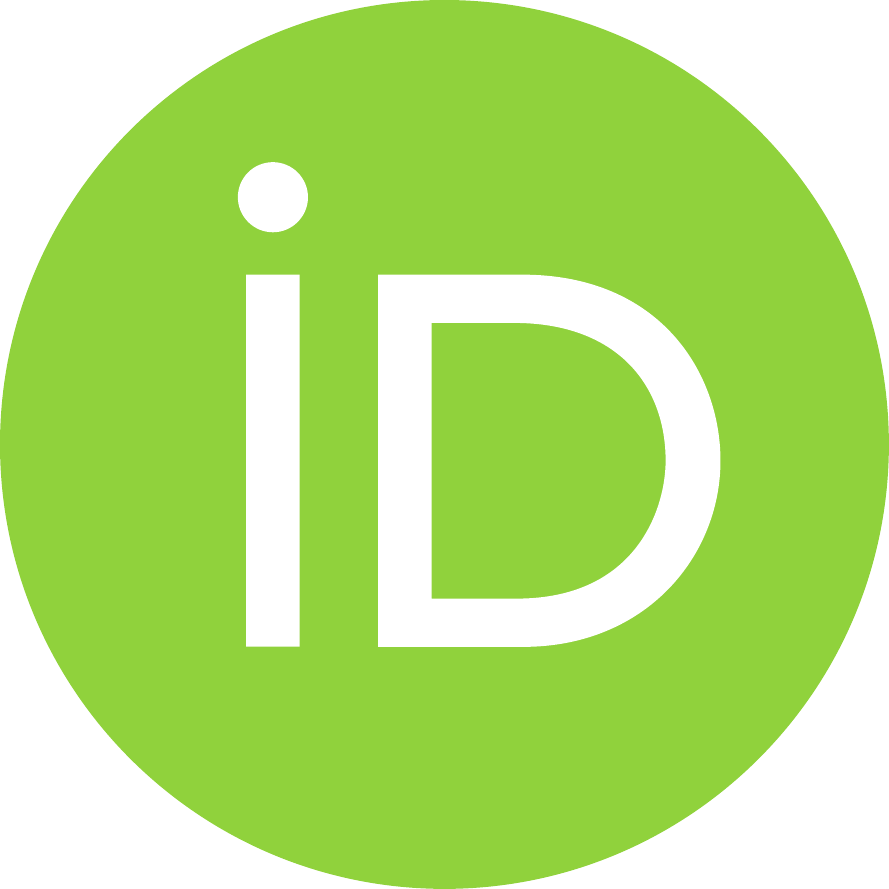} \url{https://orcid.org/#1}}
\author{Cyril Banderier$^1$\thanks{\url{https://lipn.fr/~banderier} \hfill \addorcid{0000-0003-0755-3022} } 
  \and Florian Luca$^{2,3,4}$\thanks{\href{https://scholar.google.com/citations?user=g-eDxl0AAAAJ}{https://scholar.google.com/} \hfill \addorcid{0000-0001-8581-449X} }
}	

\maketitle
\thispagestyle{empty}

\begin{center}
\noindent $^1$: LIPN (UMR CNRS 7030), Universit\'e Paris Nord, France. \\
$^2$: School of Mathematics, University of the Witwatersrand, South Africa.\\
$^3$: King Abdulaziz University, Jeddah, Saudi Arabia.\\
$^4$: Department of Mathematics, University of Ostrava, Czech Republic. \\
\end{center}

\begin{abstract}
Polynomially-recursive sequences generally have a periodic behavior mod $m$.
In this paper, we analyze the period mod $m$ of a second order polynomially-recursive sequence.
The problem originally comes from an enumeration of avoiding pattern permutations
and appears to be linked with nice number theory notions  (the Carmichael function, Wieferich primes, algebraic integers).
We give the mod $a^k$ supercongruences, and generalize these results to a class of recurrences.
\medskip


\end{abstract}

\newpage
\section{Introduction}\label{sec1}

In his analysis of sorting algorithms, Knuth introduced the notion of forbidden pattern 
in permutations, which later became a field of interest per se~\cite{Knuth98}. 
By studying the basis of such forbidden patterns for 
permutations reachable with $k$ right-jumps from the identity permutation, 
the authors of~\cite{BanderierBarilMoreiraDosSantos17} discovered that the permutations of size $n$ in this basis 
were enumerated by the sequence of integers $\{b_n\}_{n\ge 0}$  given by $b_0=1,~b_1=0$, \begin{equation}\label{rec}
b_{n+2}=2n b_{n+1}+(1+n-n^2)b_n\qquad {\text{\rm for~all}}\quad n\ge 0.
\end{equation}
This is sequence \OEIS{A265165} in the OEIS\footnote{On-Line Encyclopedia of Integer Sequences, \href{https://oeis.org}{https://oeis.org}.},
it starts like 	0, 1, 2, 7, 32, 179, 1182, 8993, 77440, 744425, 7901410, 91774375\dots

\smallskip
Such a sequence satisfying a recurrence with polynomial coefficients in $n$ is called \mbox{{\em P-recursive}} (for {\em polynomially recursive}), D-finite, or holonomic,
depending on the authors (see e.g.~\cite{Stanley99,FlajoletSedgewick09,Gessel90,PetkovsekWilfZeilberger96}). P-recursive sequences are ubiquitous in combinatorics, number theory, analysis of algorithms, computer algebra, etc.
It is always the case that the corresponding generating function satisfies a linear differential equation, 
but it is not always the case that it has a closed form. The generating function of  $\{b_n\}_{n\ge 0}$  has in fact a nice closed form involving the golden ratio. Indeed, putting 
$$
(\A,\B):=\left(\frac{1+{\sqrt{5}}}{2},\frac{1-{\sqrt{5}}}{2}\right)
$$
for the two roots of the quadratic equation $x^2-x-1=0$, it was shown in \cite{BanderierBarilMoreiraDosSantos17} that the exponential generating function of the $\{b_n\}_{n\ge 0}$,
namely 
\begin{equation}
\label{eq:B}
B(x)=\sum_{n\ge 0} b_n\frac{x^n}{n!},\qquad {\text{satisfies}}\quad B(x)=\frac{-\B}{\A-\B} (1-x)^{\A}+\frac{\A}{\A-\B} (1-x)^{\B}-1.
\end{equation}
This is a noteworthy sequence in analytic combinatorics 
(see~\cite{FlajoletSedgewick09} for a nice presentation of this field), as it is one of the rare sequences exhibiting an irrational exponent in its asymptotics: 
$$
\frac{b_n}{n!}\sim \frac{\A}{ \sqrt{5} \Gamma(\A-1)}   {n^{\A-2}}(1+o(1))\qquad {\text{\rm as}}\quad n\to\infty,
$$
where $\Gamma(z)=\int_0^{+\infty} t^{z-1}\exp(-t)dt$ is the Euler gamma function. 

\pagebreak

There is a vast literature in number theory analyzing the modular congruences of famous sequences (Pascal triangle, Fibonacci, Catalan, Motzkin, Ap\'ery numbers, see~\cite{Gessel84,DeutschSagan06, XinXu11, RowlandZeilberger14, KauersKrattenthalerMueller11}).
The properties of $b_n \mod m$ are sometimes called ``supercongruences'' when $m$ is the power of a prime number: many
articles consider $m=2^r$, or $m=3^r$.
We now restate an important result which holds for any $m$ (not necessarily the power of a prime number).

\begin{theorem}[Supercongruences for D-finite functions, Theorem 7 of \cite{BanderierBarilMoreiraDosSantos17}]\qquad\newline   \label{supercongru}
Consider  any P-recurrence of order $r$: $$P_0(n) u_n = \sum_{i=1}^r  P_i(n) u_{n-i}\,,$$
where the polynomials $P_0(n),\dots, P_r(n)$ belong to  ${\mathbb Z}[n]$,
and where the polynomial $P_0(n)$ is ultimately invertible ${\rm mod\ } m$ (\/{\em i.e.,} $\gcd(P_0(n),m) =1 $ for all $n$ large enough).
Then the sequence $(u_n)$  is eventually periodic\footnote{In the sequel, we will omit the word ``eventually'': a periodic sequence of period $p$ is thus a sequence for which  $u_{n+p} = u_n$ for all large enough $n$. Some authors use the terminology ``ultimately periodic'' instead.}
 ${\rm mod\ } m$. In particular, recurrences such that $P_0(n)=1$ are periodic ${\rm mod\ } m$.
Additionally, the period is always bounded by $m^{2r}$, therefore there is an algorithm to compute it.
\end{theorem}

\smallskip

N.B.: It is not always the case that P-recursive sequences are periodic mod $p$. E.g.,
it was proven in~\cite{KlazarLuca05} that Motzkin numbers are not periodic mod $m$,
and it seems that
$$(n+3) (n+2) u_n=8 (n-1)(n-2) u(n-2)+(7 n^2+7 n-2) u(n-1)\,, \quad u_0=0, u_1=1\,,$$
is also not periodic mod $m$, for any $m>2$
(this P-recursive sequence counts a famous class of permutations, namely, the Baxter permutations).
This is coherent with Theorem~\ref{supercongru},
as  the leading term in the recurrence (the factor $(n+3)(n+2)$)  is not invertible mod $m$, for infinitely many $n$.

\bigskip

For our sequence $\{b_n\}_{n\ge 1}$ (defined by recurrence~\eqref{rec}),
this theorem explains the periodic behavior of $b_n {\ \rm mod\ } m$.
By brute-force computation, we can get $b_n {\ \rm mod\ } m$, for any given $m$.
For example $b_n {\ \rm mod\ } 15$ is periodic of period $12$:
$$\{b_n {\ \rm mod\ }  15\}_{n\ge 9}= (10, 5, 10, 10, 0, 10, 5, 10, 5, 5, 0, 5)^\infty.$$
The period can be quite large, for example $b_n {\ \rm mod\ } 3617$ has period $26158144$.
More generally, for every positive integer $m$, the sequence $\{b_n {\ \rm mod\ }  m \}_{n\ge 1}$ is eventually periodic: there exist
$T_m>0$ and $n_m$ such that, for all $n\ge n_m$, one has $b_{n+T_m}\equiv b_n\pmod {m}$. 
We write $T_m$ for the smallest such period. In this paper, we study some of the properties of  $\{T_m\}_{m\ge 1}$.  

This is sequence \OEIS{A306699} in the OEIS,
here are its first values $T_2,\dots,T_{100}$:\\ {\footnotesize
2, 12, 8, 1, 12, 84, 8, 36, 2, 1, 24, 104, 84, 12, 16, 544, 36, 1, 8, 84, 2, 1012, 24, 1, 104, 108, 168, 1, 12, 1, 32, 12, 544, 84, 72, 2664, 2, 312, 8, 1, 84, 3612, 8, 36, 1012, 4324, 48, 588, 2, 1632, 104, 5512, 108, 1, 168, 12, 2, 1, 24, 1, 2, 252, 64, 104, 12, 2948, 544, 3036, 84, 1, 72, 10512, 2664, 12, 8, 84, 312, 1, 16, 324, 2, 13612, 168, 544, 3612, 12, 8, 1, 36, 2184, 2024, 12, 4324, 1, 96, 18624, 588, 36, 8.}\\
\noindent Do you detect the hidden patterns in this sequence? 
This is what we tackle in the next section.

\newpage
\section{Periodicity mod $m$, supercongruences and links with number theory}

Our main result is the following.

\begin{theorem}\label{theo1}
Let $b_n$ be the sequence defined by the recurrence of Formula~\ref{rec}.  The period $T_m$ of this sequence $b_n$ mod $m$ satisfies:
\begin{itemize}
\item[a)]
 If $m=p_1^{e_1}\ldots p_k^{e_k}$ (where $p_1,\ldots,p_k$ are distinct primes), then\footnote{As usual, lcm stands for the {\em least common multiple}.} 
$$
T_m=\lcm(T_{p_1^{e_1}},\ldots,T_{p_k^{e_k}}).
$$
\item[b)]
We have $T_m=1$ if and only if $m$ is the product of primes $p\equiv0,1,4\pmod 5$.
\item[c)] 
For every prime $p$, we have $T_p\mid 2p\, \ord_5(p)$  (and thus $T_p\mid 2p(p-1)$).
\item[d)] 
If $T_m>1$ then $T_m$ is even (and a multiple of 4 if $m$ is prime).
\item[e)] 
For $m\geq 3$, we have $T_m=2$ if and only if $m$ is even and $\frac{m}{2}$ is the product of primes $p\equiv 0,1,4\pmod 5$.
\item[f)]
For any prime $p$ not $0,1,4\pmod 5$, we have $T_{p^k}\mid 2p^k (p-1)$.
\end{itemize}
\end{theorem}

The function $T_m$ thus shares some similarities with the Carmichael function introduced in~\cite[p.~39]{Carmichael14},
and it is expected that its asymptotic behavior is also similar (following e.g.~the lines of~\cite{ErdosPomeranceSchmutz91}).
In this article, we focus on the rich arithmetic properties of this function.  
Note that it allows to compute $T_m$ in a much faster way than the 
brute-force algorithm mentioned  in Section~\ref{sec1}:
the complexity goes from $m^{2r}$ via brute-force to 
e.g.~$\ln(m)^3$ via Shor's algorithm
(or some other sub-exponential complexity in $\ln(m)$ with other efficient algorithms).

\bigskip
{\bf Proof of Part a)}. The proof will use a little preliminary result and the following definition. 
We call $T_m$ the ``eventual period of the sequence mod $m$'', or for short with a
slight abuse of terminology, the ``period of the sequence mod $m$'' (even if the sequence starts with some terms which does not satisfy the periodic pattern).  
The following lemma holds for all eventually periodic sequences of integers.

\begin{lemma}
$T_m$ divides all other periods of $\{u_n\}_{n\ge 0}$ modulo $m$. 
\end{lemma} 

\begin{proof}
Let $T_m=a$ and assume there is $b$ (not a multiple of $a$) which is also a period modulo $m$. Thus, there are $n_a,~n_b$ such that $u_{n+a}\equiv u_n\pmod m$ for all $n>n_a$ and 
$u_n\equiv u_{n+b}\pmod m$ for all $n>n_b$. Let $d=\gcd(a,b)$. 
By B{\'e}zout's identity, one has then $d=Aa+Bb$ for some integers $A,~B$. Let $n_{a,b}=\max\{n_a,n_b\}+|A|a+|B|b$ 
and assume that $n>n_{a,b}$. Then  $u_{a+d}=u_{n+Aa+Bb}\equiv u_{(n+Aa)+bB}\pmod m\equiv u_{n+Aa}\pmod m\equiv u_{n}\pmod m$ so $d<a$ is a period of $\{u_n\}_{n\ge 0}$ modulo $m$, contradicting 
the minimality of $a$. 
\end{proof}

\pagebreak
An immediate consequence is the following\footnote{We use the notation $[m_1,\dots,m_r]=\lcm(m_1,\dots,m_r)$ for the least common multiple of integers $m_1,\dots,m_r$.}:

\begin{corollary}
\label{cor1}
We have $T_{[m_1,\ldots,m_r]}=[T_{m_1},\ldots,T_{m_r}]$.
\end{corollary}
\begin{proof}
First consider $r=2$, and let $a:=m_1$, $b:=m_2$.
Since $[T_a,T_b]$ is a multiple of both $T_a$ and $T_b$, it follows that it is a period of $\{u_n\}_{n\ge 0}$ modulo both $a$ and $b$, so modulo $[a,b]$. It remains to prove that it is the minimal one. To this aim, suppose that $T_{[a,b]}<[T_a,T_b]$.
Then either $T_a\nmid T_{[a,b]}$ or $T_b\nmid T_{[a,b]}$. Since the two cases are similar, we only deal with the first one. 
In this  case we would have that both $T_a$ and $T_{[a,b]}$ would be  periods modulo~$a$.
By the previous lemma, this would force $\gcd(T_a,T{[a,b]})<T_a$, which would obviously be a contradiction. 
Now, a trivial induction on the number $r\ge 2$ gives that
$$
T_{[m_1,\ldots,m_r]}=[T_{m_1},\ldots,T_{m_r}]
$$
holds for all positive integers $m_1,\ldots,m_r$. 
\end{proof}

In particular Part a) of Theorem~\ref{theo1} holds: $T_m=\lcm(T_{p_1^{e_1}},\ldots,T_{p_k^{e_k}})$.  Let us now tackle the proofs of Parts b)--f).

\bigskip
{\bf Proof of Part b).} We use the generating function \eqref{eq:B}, which tells us that
\begin{equation}
[x^n] B(x)   =  \frac{b_n}{n!}=\frac{(-1)^n}{\sqrt{5}}\left(\A\binom{\B}{n}-\B \binom{\A}{n}\right).
\end{equation}
Thus, 
\begin{equation}
b_n =  \frac{(-1)^{n-1}}{\sqrt{5}} \left(\B \A(\A-1)\cdots (\A-(n-1))-\A \B(\B-1)\cdots (\B-n+1)\right).
\end{equation}
By Fermat's little theorem, 
\begin{equation}
\label{eq:X}
\prod_{k=0}^{p-1} (X-k)=X^p-X\pmod p.
\end{equation}
Assume now that $p\equiv 1,4\pmod 5$. Then 
$$
\prod_{k=0}^{p-1} (\A-k)\equiv \A^p-\A\pmod p\equiv 0\pmod p,
$$
where for the last congruence we used the law of quadratic reciprocity:
since $p\equiv 1,4\pmod 5$, we have $${\left(\frac{5}{p}\right)=\left(\frac{p}{5}\right)=1},$$ where 
 ${\displaystyle{\left(\frac{\bullet}{p}\right)}}$ is the Legendre symbol. Thus, 
\begin{equation}
\label{eq:alpha}
\A^p=\left(\frac{1+{\sqrt{5}}}{2}\right)^p\equiv \frac{1+{\sqrt{5}}\cdot 5^{(p-1)/2}}{2^p}\pmod p\equiv \A \pmod p,
\end{equation}
because $5^{(p-1)/2}\equiv {\displaystyle{\left(\frac{5}{p}\right)}}\equiv 1\pmod p$ by Euler's criterion. 
\clearpage

In the above and in what follows, for two algebraic integers $\delta,~\gamma$ and an integer $m$ we write $\delta\equiv \gamma\pmod m$ if the number $(\delta-\gamma)/m$ is an algebraic integer.
This shows that
$$
\frac{1}{p} \prod_{k=0}^{p-1} (\A-k)
$$
is an algebraic integer. The same is true with $\A$ replaced by $\B$. Now take $r\ge 1$ be any integer and take $n\ge pr$. Then, for each $\ell=0,1,\ldots,r-1$, we have that both
$$
\frac{1}{p}\prod_{k=0}^{p-1} (\A-(p\ell+k))\quad {\text{\rm and }}\qquad \frac{1}{p}\prod_{k=0}^{p-1} (\B-(p\ell+k))
$$
are algebraic integers. Thus, if $n\ge pr$, then 
$$
\frac{{\sqrt{5}} b_n}{p^r}=(-1)^{n-1}\left(\B\prod_{\ell=0}^{r-1}\prod_{k=0}^{p-1} (\A-(p\ell+k))\prod_{k=pr}^{n-1} (\A-k)-\A\prod_{\ell=0}^{r-1}\prod_{k=0}^{p-1} (\B-(p\ell+k))\prod_{k=pr}^{n-1} (\B-k)\right)
$$
is an algebraic integer. Thus, $5b_n^2/p^{2r}$ is an algebraic integer and a rational number, so an integer. Since $p\ne 5$, it follows that $p^{2r}\mid b_n^2$, so $p^r\mid b_n$ for $n\ge pr$. This shows that
$T_{p^r}=1$ for all such primes $p$ and positive integers $r$. The same is true for $p=5$. There we use that $\A-3={\sqrt{5}}\B$, so ${\sqrt{5}}\mid \A-3$. Thus, if $n\ge 10r$, we have that 
$$
\prod_{k=1}^n (\A-k)\quad {\text{\rm is~a~multiple~of}}\quad \prod_{\ell=0}^{2r-1} (\A-(3+5\ell))\quad {\text{\rm in}}\quad {\mathbb Z}[(1+{\sqrt{5}})/2],
$$
which in turn is a multiple of $5^r={\sqrt{5}}^{2r}$ in ${\mathbb Z}[(1+{\sqrt{5}})/2]$. Thus, if $n\ge 10r$, then $5^r\mid b_n$. This shows that also $T_{5^r}=1$ and in fact, 
$m\mid b_n$ for all $n>n_m$ if $m$ is made up only of primes $0,1,4\pmod 5$. This finishes the proof of b).

\bigskip
{\bf Proof of Part c).} The claim is satisfied for $p=2$,
as $\{b_n\ {\rm mod\ } 2\}_{n\geq 0}=(1,0)^\infty$, thus $T_2=2 |4$.
Now consider $p>2$. 
Evaluating Formula~\eqref{eq:X} at~$\A=\frac{1+\sqrt{5}}{2}$, one has
$$
\prod_{k=0}^{p-1} (\A-k)\equiv \A^p-\A\pmod p. 
$$
Since $5^{(p-1)/2}\equiv -1\pmod p$, the argument from \eqref{eq:alpha} shows that $\A^p\equiv \B\pmod p$. Thus
$$
\prod_{k=1}^{2p} (\A-k)=\prod_{k=1}^p (\A-k) \prod_{k=p+1}^{2p} (\A-k)\equiv (\B-\A)^2\pmod p\equiv 5\pmod p.
$$
The same is true for $\A$ replaced by $\B$. Thus, it follows that for $n>2p$, we have
\begin{eqnarray*}
b_{n+2p} & = & \frac{(-1)^{n+2p-1}}{{\sqrt{5}}} \left(\B \prod_{k=0}^{n+2p-1} (\A-k)-\A \prod_{k=0}^{n+2p-1} (\B-k)\right)\\
& \equiv & \frac{(-1)^{n-1}}{{\sqrt{5}}} 5 \left(\B\prod_{k=0}^{n-1} (\A-k)-\A \prod_{k=0}^{n-1} (\B-k)\right)\pmod p\\
& \equiv &  5b_n\pmod p.
\end{eqnarray*}  
Applying this $k$ times, we get
$$
b_{n+2pk}\equiv 5^k b_n\pmod p.
$$
Taking $k=p-1$ and applying Fermat's little theorem $5^{p-1}\equiv 1\pmod p$, 
we get $T_p\mid 2p(p-1)$.
In fact, taking $k=\ord_p(5)$,
 where $\ord_p(5)$ is the order of $5$ modulo $p$ (the smallest $k>0$ such that $5^k \equiv 1\pmod p$)
gives the slightly stronger claim: $T_p\mid 2p\, \ord_p(5)$.

\bigskip
{\bf Proof of Part d).}
There are more things to learn from the above argument. We first prove by contradiction
the second claim of d): $4|T_p$, for a prime $p$ such that $T_p>1$.
Assume 
$\nu_2(T_p)<2$,
where $\nu_q(a)$ is the exponent of $q$ in the factorization of $a$. That is, $T_p$ is either odd or $2$ times an odd number. Since $T_p\mid 2p(p-1)$, it follows that if we write $p-1=2^a k$, where $k$ is odd, then $T_p\mid 2pk$. Thus,
\begin{equation}\label{eqmodp}
b_{n}\equiv b_{n+2pk}\equiv 5^k b_n\pmod p
\end{equation}
for all $n>n_p$. Since $5$ is not a quadratic residue,  it follows that $5^k\not\equiv 1\pmod p$ (since $-1\equiv 5^{(p-1)/2}\equiv (5^k)^{2^{a-1}}\pmod p$). So, the above congruence~\eqref{eqmodp} implies that $p\mid (5^k-1)b_n$ but
$p\nmid 5^k-1$, so $b_n\equiv 0\pmod p$ for all large $n$. Take $n$ and $n+1$ and rewrite the information that $b_n\equiv b_{n+1}\equiv 0\pmod p$ in ${\mathbb Z}[\A]/p{\mathbb Z}[\A]$ as 
\begin{eqnarray*}
b_n=\B \prod_{k=0}^{n-1} (\A-k)-\A \prod_{k=0}^{n-1} (\B-k) & \equiv &  0\pmod p \,,\\
b_{n+1}=\B \left(\prod_{k=0}^{n-1} (\A-k)\right)(\A-n)-\B\left(\prod_{k=0}^{n-1}(\B-k)\right)(\B-n) & \equiv & 0\pmod p.
\end{eqnarray*}
We treat this as a linear system in the two unknowns 
$$
(X,Y)=\left(\B\prod_{k=0}^{n-1}(\A-k),\A \prod_{k=0}^{n-1} (\B-k)\right)
$$ 
in the field with $p^2$ elements ${\mathbb Z}[\A]/p{\mathbb Z}[\A]$. This is homogeneous. None of $X$ or $Y$ is $0$ since $p$ cannot divide $\B\prod_{k=0}^{n-1}(\A-k)$. Thus, it must be that the determinant of the above 
matrix is $0$ modulo $p$, but this is 
$$
\left|\begin{matrix}1 & -1 \\ \A-n & -(\B-n)\end{matrix}\right|={\sqrt{5}},
$$
which is invertible modulo $p$. Thus, indeed, it is not possible that $b_n$ and $b_{n+1}$ is a multiple of $p$ for all large $n$, getting a contradiction. 
This shows that $T_p$ is a multiple of $4$. 

\bigskip
{\bf Proof of Part e) (and first claim in Part d).}
Now let $m$ which is not like in b), i.e.~one has at least one prime $p\equiv 2,3\pmod 5$ such that $p\mid m$. Then $4\mid T_p$ by what we have done above, and so $4\mid T_m$ by a). Thus, such $m$ cannot participate in
the situations described either at  d) or e). Further, one has $T_4=8$ as 
$\{b_n \ {\rm mod}\ 4 \}_{n\ge 0} = (1,0,1,2,3,0,3,2)^\infty$.
Thus, if $4\mid m$, then $8\mid T_m$. Hence, if $T_m=2$, then the only possibility is that $2\mid m$ and $m/2$ is a product of primes congruent to $0,1,4$ modulo $5$.
Conversely, if $m$ has such structure then $T_m=2$ by a) and the fact that $T_2=2$ and $T_{p^r}=1$ for all odd prime power factors $p^r$ of $m$. 
This ends the proof of e) and d). 

{\bf Proof of Part f).} Finally, f) is based on a slight generalization of \eqref{eq:X} namely
\begin{equation}
\label{eq:ptor}
\prod_{k=0}^{p^r-1} (X-k)\equiv (X^p-X)^{p^{r-1}}\pmod {p^r}
\end{equation}
valid for all odd primes $p$ and $r\ge 1$. Let us prove \eqref{eq:ptor}. We first prove it for $r=2$. We return to \eqref{eq:X} and write
$$
\prod_{k=0}^{p-1} (X-k)=X^p-X+pH_1(X),
$$
where $H_1(X)\in {\mathbb Z}[X]$. Changing $X$ to $X-p\ell$ for $\ell=0,1,\ldots,p-1$, we get that
$$
\prod_{k=0}^{p-1} (X-(p\ell+k))=(X-p\ell)^p-(X-p\ell)+pH(X-p\ell)\equiv (X^p-X-pH(X))-p\ell\pmod {p^2}.
$$
In the above, we used the fact that $H(X-p\ell)\equiv H(X)\pmod p$. Thus,
\begin{eqnarray*}
\prod_{k=0}^{p^2-1} (X-k) & = & \prod_{\ell=0}^{p-1} \prod_{k=0}^{p-1} (X-(p\ell+k))\\
 & \equiv & \prod_{k=0}^{p-1} ((X^p-X-pH(X))-p\ell)\pmod {p^2}\\
 & \equiv & (X^p-X-pH(X))^p-(X^p-X-pH(X))^{p-1} p\left(\sum_{\ell=0}^{p-1} \ell\right)\pmod {p^2}\\
 & \equiv & (X^p-X)^p-(X^p-X-pH(X))^{p-1} p\left(\frac{p(p-1)}{2}\right)\pmod {p^2}\\
 & \equiv & (X^p-X)^p\pmod {p^2}.
 \end{eqnarray*}
 In the above, we used the fact that $p$ is odd so $p(p-1)/2$ is a multiple of $p$. This proves \eqref{eq:ptor} for $r=2$. Assuming $r\ge 2$ and that \eqref{eq:ptor} holds for $p^r$, we get that for all $\ell\ge 0$, we have
\begin{eqnarray*}
 \prod_{k=0}^{p^r-1} (X-(p^r\ell+k)) & \equiv & ((X-p^r\ell)^p-(X-p^r\ell))^{p^{r-1}}+p^rH_r(X-p^r\ell)\pmod {p^{r+1}}\\
 & \equiv & (X^p-X)^{p^{r-1}}+p^rH_r(X)\pmod {p^{r+1}},
 \end{eqnarray*}
 where $H_r(X)\in {\mathbb Z}[X]$. Thus, 
 \begin{eqnarray*}
 \prod_{k=0}^{p^{r+1}-1} (X-k) & = & \prod_{\ell=0}^p \prod_{k=0}^{p^r-1} (X-(p^r\ell+k))\\
 & \equiv & ((X^p-X)^{p^{r-1}}+p^r H_r(X))^p \pmod {p^{r+1}} \\
 & \equiv & (X^p-X)^{p^r}\pmod {p^{r+1}},
 \end{eqnarray*}
 which is what we wanted. 
Letting $p>2$ be congruent to $2,3\pmod 5$, we and evaluating the above in $\A$ and using that $\A^p\equiv \B\pmod p$, we get easily that
$$
\prod_{k=0}^{p^r-1} (\A-k)\equiv (X^p-X)^{p^{r-1}}\pmod {p^r}\equiv (\A^p-\A)^{p^{r-1}}\pmod {p^r}\equiv (\B-\A)^{p^{r-1}}\pmod {p^r}.
$$
This shows that
$$
\prod_{k=0}^{2p^r-1} (\A-k)\equiv (\B-\A)^{2p^{r-1}}\pmod {p^r}\equiv 5^{p^{r-1}}\pmod {p^r}.
$$
The same is true for $\B$ leading to 
$$
b_{n+2p^r}\equiv \frac{(-1)^{n+2p^r-1}}{{\sqrt{5}}} 5^{p^{r-1}} \left(\B\prod_{k=0}^{n-1}(\A-k)-\A\prod_{k=0}^{n-1} (\B-k)\right)\pmod {p^r}\equiv 5^{p^{r-1}} b_n\pmod {p^r}.
$$
Thus, applying this $k$ times we get
$$
b_{n+2p^r k} \equiv 5^{p^{r-1}k} b_n\pmod {p^r}.
$$
Taking $k=p-1$ and applying Euler's theorem $5^{p^{r-1}(p-1)}\equiv 1\pmod {p^r}$, we get that $b_{n+2p^{r}(p-1)}\equiv b_n\pmod {p^r}$. Thus, $T_{p^r}\mid 2p^r (p-1)$. As in c), we can replace 
$p-1$ by $\ord_p(5)$ and the divisibility holds.  

Finally, it remains to prove f) for $p=2$. Here, by inspection, we have 
$$
\prod_{k=0}^7 (X-k)\equiv (X^2-X)^4\pmod 4.
$$
 By induction on $r\ge 2$, one shows that
$$
\prod_{k=0}^{2^{r+1}-1}(X-k)\equiv (X^2-X)^{2^r}\pmod {2^r}.
$$
Evaluating this in $\A$, we get  
$$
\prod_{k=0}^{2^{r+1}-1} (\A-k)\equiv (\A^2-\B)^{2^r}\equiv 5^{2^{r-1}}\pmod {2^r}.
$$
The same holds for $\B$, so 
\begin{eqnarray*}
b_{n+2^{r+1}} & = & \frac{(-1)^{n+2^{r+1}-1}}{\sqrt{5}} 5^{2^{r-1}} \left(\B\prod_{k=0}^{n-1}(\A-k)-\A\prod_{k=0}^{n-1} (\B-k)\right)\pmod {2^r}\\
& \equiv & 5^{2^{r-1}} b_n\pmod {2^r}\equiv b_n\pmod {2^r}
\end{eqnarray*}
showing that $T_{2^r}\mid 2^{r+1}$ for all $r\ge 2$. 

\pagebreak
\section{Comments and generalizations}

Along the proof of our main result we showed that if $p\equiv 2,3\pmod 5$, then 
$$
b_{n+2p}\equiv 5b_{n}\pmod p.
$$
From here we deduced that $T_p\mid 2p(p-1)$ via the fact that $5^{p-1}\equiv 1\pmod p$. One may ask whether it can be the case that $T_{p^2}\mid 2p(p-1)$ for some prime $p$. 
Well, first of all, we will need that $5^{p-1}\equiv 1\pmod {p^2}$. This makes $p$ a base $5$ Wieferich prime. There is a conjecture that there are infinitely many such primes. The smallest known 
which is also congruent to $2,3\pmod 5$ is $40487$. However, the condition of condition of $p$ being base $5$ Wieferich is not sufficient. A close analysis of our arguments show that 
in addition to this condition, it should also hold that
$$
\prod_{k=0}^{2p-1}(\alpha-k)-5\equiv 0\pmod {p^2}\,,
$$
 and if this is the case then indeed $T_{p^2}\mid 2p(p-1)$. Since  the integer $(1/p)(\prod_{k=0}^{2p-1} (\alpha-k)-5)$ in ${\mathbb Z}[\A]$ should be the zero element in the finite field ${\mathbb Z}[\A]/p{\mathbb Z}[\A]$, with $p^2$ elements, it could be that the ``probability'' that this condition happens 
is $1/p^2$. By the same logic, the ``probability'' that $p$ is base $5$ Wieferich  should be $1/p$. Assuming these events to be independent, we could infer that the probability that both these conditions hold is $1/p^3$ and the series 
$$
\sum_{p\equiv 2,3\pmod 5} \frac{1}{p^3}
$$
is convergent, which seems to suggest, heuristically, that  there should be only finitely many primes $p\equiv 2,3\pmod 5$ such that $T_{p^2}\mid 2p(p-1)$.

Finally, our results apply to other sequences as well. More precisely, let $a,~b$ be integers and let $\alpha,~\beta$ be the roots of $x^2-ax-b$. Let 
$$
B_{a,b}=\frac{-\beta}{\alpha-\beta} (1-x)^{\alpha}+\frac{\alpha}{\alpha-\beta} (1-x)^{\beta}=\sum_{n\ge 0} b_n \frac{x^n}{n!}.
$$
The sequence $\{b_n\}_{n\ge 0}$ satisfies $b_0=1$, $b_1=0$, and, for $n\ge 0$
$$
b_{n+2}=(2n-a+1)b_{n+1}+(b+an-n^2)b_n.
$$
In case $\alpha$ and $\beta$ are rational (hence, integers), $B(x)$ is a rational function, so $b_n=n! u_n$, where $\{u_n\}_{n\ge 0}$ is binary recurrent with constant coefficients. It then follows that 
$b_n\equiv 0\pmod m$ for all $m$ provided $n>n_m$ is sufficiently large. Thus, $T_m=1$. In case $\alpha,\beta$ are irrational, then a similar result holds as for the case when $(a,b)=(1,1)$. Namely, 
$b_n\equiv 0\pmod m$ for all $n$ sufficiently large whenever $m$ is the product of odd primes $p$ for which ${\displaystyle{\left(\frac{\Delta}{p}\right)=0,1}}$, where $\Delta=a^2+4b$ is the discriminant of 
the quadratic $x^2-ax-b$. In case $p$ is odd and ${\displaystyle{\left(\frac{\Delta}{p}\right)=-1}}$, we have that $T_p\mid 2p(p-1)$ and $T_p$ is a multiple of $4$. Also, $T_{p^r}\mid 2p^r(p-1)$ 
for all $r\ge 1$ in this case. The proofs are similar. In the case of the prime $2$ one needs to distinguish cases according to the parities of $a,b$. For example, if $a$ and $b$ are odd,
then $\Delta\equiv 5\pmod 8$, so $2$ is not a quadratic residue modulo $\Delta$, so $T_{2^r}\mid 2^{r+1}$ for all $r\ge 1$, whereas if $a$ is odd and $b$ is even then $T_2=1$. 
This concludes our analysis of the periodicity of such P-recursive sequences ${\rm mod\ } m$.

\bigskip
{\bf Acknowledgments:}
This work was initiated when the second author was invited professor at the Paris Nord University, in April 2018. 
In addition, Florian Luca was supported in part by grant CPRR160325161141 and an A-rated scientist award 
both from the NRF of South Africa and by grant no.~17-02804S of the Czech Granting Agency.  
\noindent
\vspace{-3mm}

\let\OLDthebibliography\thebibliography
\renewcommand\thebibliography[1]{
  \OLDthebibliography{#1}
  \setlength{\parskip}{0pt}
  \setlength{\itemsep}{3pt}
}

\bibliographystyle{cyrbiburl}
\linespread{.94}\selectfont

\end{document}